\documentclass[12pt,reqno]{amsart}

\usepackage[text={400pt,660pt},centering]{geometry}

\usepackage{color}
\usepackage{comment}
\usepackage{esint,amssymb}
\usepackage{graphicx}
\usepackage{MnSymbol}
\usepackage{mathtools}
\usepackage[colorlinks=true, pdfstartview=FitV, linkcolor=blue, citecolor=blue, urlcolor=blue,pagebackref=false]{hyperref}
\usepackage{microtype}

\usepackage{bm}
\usepackage{scalerel} %for scaling the boxes 
\usepackage{dsfont}
\usepackage{mathrsfs}
\usepackage{eufrak}
\usepackage[font={footnotesize}]{caption}
\usepackage{comment}
\usepackage[shortlabels]{enumitem}
\usepackage{bbm}
\usepackage[normalem]{ulem}
\usepackage[numbers,sort&compress]{natbib}

\usepackage{xparse}

\newtheorem{theorem}{Theorem}[section]
\newtheorem{proposition}{Proposition}
\newtheorem{lemma}[proposition]{Lemma}
\newtheorem{corollary}[proposition]{Corollary}

\theoremstyle{remark}
\newtheorem{remark}[proposition]{Remark}

\theoremstyle{definition}

\numberwithin{equation}{section}
\numberwithin{proposition}{section}

\newcommand{\NN}{\mathbb{N}}
\newcommand{\ZZ}{\mathbb{Z}}

\newcommand{\RR}{\mathbb{R}}

\newcommand{\e}{\varepsilon}

\newcommand{\al}{\alpha}

\newcommand{\ga}{\gamma}
\newcommand{\de}{\delta}

\newcommand{\la}{\lambda}

\newcommand{\sig}{\sigma}

\DeclareMathOperator{\tr}{tr}

\DeclareMathOperator{\cov}{Cov}
\DeclareMathOperator{\normal}{\mathcal{N}}

\newcommand{\EE}{\mathbb{E}}
\DeclareDocumentCommand{\E}{o m}  
{%
	\mathbb{E}\IfValueT{#1}{_{#1}}\left[#2\right]
}
\DeclareDocumentCommand{\P}{o m}  
{%
	\mathbb{P}\IfValueT{#1}{_{#1}}\left[#2\right]
}

\newcommand{\tmesh}{{n^{-1}}}
\newcommand{\xmesh}{{n^{-1/2}}}

\renewcommand{\tilde}{\widetilde}

\newcommand{\norm}[1]{\|#1\|}

\makeatletter
\DeclareFontFamily{OMX}{MnSymbolE}{}
\DeclareSymbolFont{MnLargeSymbols}{OMX}{MnSymbolE}{m}{n}
\SetSymbolFont{MnLargeSymbols}{bold}{OMX}{MnSymbolE}{b}{n}
\DeclareFontShape{OMX}{MnSymbolE}{m}{n}{
	<-6>  MnSymbolE5
	<6-7>  MnSymbolE6
	<7-8>  MnSymbolE7
	<8-9>  MnSymbolE8
	<9-10> MnSymbolE9
	<10-12> MnSymbolE10
	<12->   MnSymbolE12
}{}
\DeclareFontShape{OMX}{MnSymbolE}{b}{n}{
	<-6>  MnSymbolE-Bold5
	<6-7>  MnSymbolE-Bold6
	<7-8>  MnSymbolE-Bold7
	<8-9>  MnSymbolE-Bold8
	<9-10> MnSymbolE-Bold9
	<10-12> MnSymbolE-Bold10
	<12->   MnSymbolE-Bold12
}{}

\let\llangle\@undefined
\let\rrangle\@undefined
\DeclareMathDelimiter{\llangle}{\mathopen}%
{MnLargeSymbols}{'164}{MnLargeSymbols}{'164}
\DeclareMathDelimiter{\rrangle}{\mathclose}%
{MnLargeSymbols}{'171}{MnLargeSymbols}{'171}
\makeatother

\newcounter{gscan}
\newcounter{bwscan}
\newcounter{cscan}
\newcounter{hscan}
\newcounter{fscan}
\newcounter{pscan}
\newcounter{sscan}
\newcounter{iscan}
\newcounter{rscan}
\newcounter{rrscan}
\newcounter{fpscan}

\newcommand{\rst}[1]{\ensuremath{{\mathbin\upharpoonright}%
		\raise-.5ex\hbox{$#1$}}}

\newcommand{\eps}{\ensuremath{\varepsilon}}
\newcommand{\eqdist}{\ensuremath{\stackrel{\mathrm{d}}{=}}}
\usepackage{setspace}
\newcommand{\R}{\ensuremath{\mathbb{R}}}
\usepackage{stmaryrd}
\newcommand{\llb}{\ensuremath{\llbracket}}
\newcommand{\rrb}{\ensuremath{\rrbracket}}

\newcommand{\mfrac}[2]{\ensuremath{\mbox{\large $\tfrac{#1}{#2}$}}}
\newcommand{\iid}{\textsc{iid}}

\begin{document}

\title{The central limit theorem via doubling of variables}

\begin{abstract}
We give a new, self-contained proof of the multidimensional central limit theorem using the technique of ``doubling variables," which is traditionally used to prove uniqueness of solutions of partial differential equations (PDEs).
Our technique also yields quantitative bounds for random variables with finite $2+\ga$ moment for some $\gamma \in (0,1]$; when $\gamma=1$, this proves a version of the Berry--Esseen theorem in $\R^d$.
\end{abstract}

\author[L. Addario-Berry]{Louigi Addario-Berry}
\author[G. Barill]{Gavin Barill}
\author[E. Beckman]{Erin Beckman}
\author[J. Lin]{Jessica Lin}

%Erin and Louigi, how would you like your names to appear? (middle names, etc.)

%60F05=Central limit and other weak theorems, 
\date{\today}

\maketitle
\section{Introduction}\label{s.intro}
The central limit theorem, under its moniker of ``the bell curve,'' is perhaps the best-known mathematical result outside of the mathematical sciences. The purpose of this paper is to provide a new, self-contained proof of the central limit theorem in $\R^d$ via {\em doubling of variables}, a key technique for establishing uniqueness of solutions of first- and second-order PDEs \cite{MR0267257, users}. 
Some of the authors of this article have recently used related techniques from the analysis of PDEs to prove distributional convergence for several random processes \cite{MR4146542,MR4391736,https://doi.org/10.48550/arxiv.2204.03689}. We believe that this family of methods will be broadly useful in probabilistic settings, and the central limit theorem seems to us a useful proving ground.

Throughout the paper we let $(X_i,i \ge 1)$ be \iid, centered, $\R^d$-valued random variables with non-degenerate, finite covariance matrix $\Sigma$, defined on a common probability space $(\Omega,\mathcal{F},\mathbb{P})$, and for integers $n \ge 1$ we set $S_n=X_1+\ldots+X_n$. 

We prove the central limit theorem in the following form.

\begin{theorem}\label{t.clt}
For all bounded, uniformly continuous $f:\R^d\to \R$, 
%any bounded, uniformly continuous function $f:\R^d \to \R$,
\begin{equation}\label{eq:convergence_of_test}
\E{f\left(\frac{S_n}{n^{1/2}}\right)} \to \E{f(\xi)}\, 
\end{equation}
as $n \to \infty$, where $\xi\sim  \normal(0,\Sigma)$. 
\end{theorem} 
The notation $\xi\sim  \normal(0,\Sigma)$ means that $\xi$ is a centered Gaussian with covariance matrix $\Sigma$. By the Portmanteau theorem, the above formulation is equivalent to the formulation in terms of cumulative distribution functions.

We write $C^k(\R^d)$ for the set of functions $f:\R^d\to \R$ for which all partial derivatives of order up to $k$ exist and are uniformly bounded. For test functions $f \in C^4(\R^d)$ which vanish at infinity, our proof technique directly yields a quantitative rate of convergence in \eqref{eq:convergence_of_test} as soon as the random variables $X_i$ have a $2+\gamma$ moment for some $\gamma > 0$.
%\lmar{Added version of quantitative statement to intro. I think it's correct, but we'll see what comes out.} 
%\jlcomment{Appears correct, although not sure if we want the dependence on $\Sigma$ or $\lambda_{max}(\Sigma)$}
%\lmar{I think we need $|f(x)|\to 0$ as $|x| \to \infty$ in Theorem~\ref{t.prohorov}?}
%\jlcomment{I think uniform continuity is automatic from the fact that $f\in C^{4}$ and $|f(x)|\rightarrow 0$ as $|x|\to \infty$, so I don't think we need to have it in the statement of the result. Also, we can remove bounded given our new def of $C^{4}$}
%\jlcomment{Need to check dependencies. Do we gain much by not taking $n\geq N$?}
\begin{theorem}\label{t.prohorov}
For all $f \in C^4(\R^d)$ with $|f(x)|\to 0$ as $|x| \to \infty$, there exists $C=C(d,\|f\|_{C^4},\Sigma)$ such that for all $\gamma \in (0,1]$, for all $n \ge 1$, 
\[
\left|\E{f\left(\frac{S_n}{n^{1/2}}\right)}-  \E{f(\xi)}\right| \leq \frac{C \E{|X_1|^{2+\gamma}|}}{n^{\gamma/2}}\, ,
\]
where $\xi\sim  \normal(0,\Sigma)$.  
\end{theorem}
The $C^4$ norm $\|f\|_{C^4}$ appearing in the statement is defined in \eqref{eq:c4norm}, below.
When $\ga= 1$, Theorem~\ref{t.prohorov} is a version of the multidimensional Berry-Esseen theorem \cite{MR3498,MR0011909}. For $\ga \in (0,1]$, this is a version of a result of Bhattacharya and Rao \cite[Theorem 18.1, Corollaries 18.2-3]{MR3396213}; however, their result allows for more general test functions $f$ and for steps which  may not be identically distributed; this leads to more complicated dependency on the tail behaviour of the $X_i$ and of~$f$. 
% \jlcomment{It appears to me that their result can handle $X_{i}$ which are not necessarily identically distributed, and the relaxed assumptions on $f$ in fact appear in Corollary 18.2. Is the latter part of the sentence just saying their estimate looks more complicated?}
 
Here is the key idea underlying the proof. 
For any bounded continuous function $f:\R^d \to \R$ and any integer $n \ge 1$, define a function $u_{n}=u_{n,f}:\R^d \times [0,\infty) \to \R$ 

%
%I am going to try and make an attempt to get all of the subscripts to look like $S_{\lfloor nt\rfloor}$, because right now we have some lone $S_{\lfloor tn\rfloor}$ scattered in there}
%\lmar{I tried to address this but another pass might be useful.}
\begin{equation}\label{eq:undef}
u_n\left(x,t\right) := \E{f\left(x+\mfrac{S_{\lfloor nt\rfloor}}{n^{1/2}}\right)}.
%= \int_{\R^d} f\left(x+\mfrac{S_{\lfloor nt\rfloor}}{n^{1/2}}\right) d\cL_{S_{\lfloor nt\rfloor}}(s)\, ,
\end{equation}
Note that $u_{n}$ is continuous in $x$, while for each $x\in \RR^{d}$, $u_{n}(x, \cdot)$ is constant on time intervals of the form $[kn^{-1}, (k+1)n^{-1})$ for $k\in \ZZ_+=\{0\}\cup\NN$. 
For $f$ and $u_n=u_{n,f}$ as above, the following proposition is straightforward.
\begin{proposition}\label{prop:recurrence}
For all $t \ge 0$ and all $x \in \R^d$, 
\begin{align}\label{e.uNscheme}
u_n(x,t+n^{-1}) 
				& = \E{u_n(x+n^{-1/2} X_{\lfloor nt\rfloor+1},t)} 
%				\\
%				& = \int u_n(x+n^{-1/2}r,t)\, \mathcal{L}(dr). \notag
\end{align}
\end{proposition}
\begin{proof}
By definition,
\[
u_n(x,t+n^{-1}) = \E{f(x+n^{-1/2} S_{\lfloor nt\rfloor+1})}
= 
\EE\left[f\left(x+ \xmesh\left(S_{\lfloor nt\rfloor} + X_{\lfloor nt\rfloor+1} \right)\right)\right].
\]
Since $X_{\lfloor nt\rfloor+1}$ is independent of $S_{\lfloor nt\rfloor}$, writing $\mathcal{L}$ for the law of $X_{\lfloor nt\rfloor +1}$ (which is also the law of $X_i$ for all $i \ge 1$), Fubini's theorem yields that 
\begin{align*}
u_n(x, t+\tmesh) &= \int_{\R^d} \EE\left[f\left(x+\xmesh r  +  \xmesh S_{\lfloor nt\rfloor} \right)\right] \, \mathcal{L}(dr) \\
&= \int_{\R^d} u_n(x + \xmesh r, t) \, \mathcal{L}(dr)\\
&= \EE\left[u_n(x + \xmesh X_{\lfloor nt\rfloor +1}, t)\right].\qedhere
\end{align*}
\end{proof}
%it is not hard to show that for all $t \ge 0$, 
%\begin{equation}\label{eq:tojustify}
%u_n\left(x,t+\mfrac{1}{n}\right) = \E{u_n\left(x+\mfrac{X_{\lfloor nt\rfloor+1}}{n^{1/2}},t\right)}; 
%\end{equation}
%see Proposition~\ref{prop:recurrence}, below. 
Subtracting $u_n(x,t)$ from each side of \eqref{e.uNscheme} and dividing by $n^{-1}$ yields
%\jlcomment{Fixed a minor typo in this statement, and then decided to switch to negative powers. I don't feel strongly about this change, but just wanted to make it correct.}
\begin{equation}\label{eq:approximate}
\frac{
u_n\left(x,t+n^{-1}\right)-
u_n\left(x,t\right)
}{n^{-1}}
=
\E{
\frac{u_n\left(x+n^{-1/2}X_{\lfloor nt\rfloor+1},t\right)
-
u_n\left(x,t\right)
}{(n^{-1/2})^2}\, 
}.
\end{equation}
The left-hand side looks like a discrete time derivative, while the right-hand side is, in fact, acting like a second-order discrete space derivative. Indeed, if $f \in C^2(\R^d)$, then also $u_n(\cdot,t) \in C^2(\R^d)$, and one can perform a Taylor approximation of $u_{n}$ on the right-hand side of \eqref{eq:approximate}. Since $X_{\lfloor nt\rfloor+1}$ is centered, the expectation then causes the first-order term of the Taylor expansion to vanish. 
More precisely, writing $D^2u_n$ to denote the Hessian of $u_n$ in the spatial coordinates, 
we show in Lemma~\ref{lem.uNspacediff} that the right-hand side closely approximates half the trace of $\Sigma D^2 u_n(x,t)$: for any $T > 0$, 
%
% indeed, a Taylor expansion of $u_{n}$ (since $f$ is smooth), and the condition that \sout{acts like} a {\em second}-order discrete space derivative; because $X_{\lfloor (n+1)t\rfloor}$ is centered, the expectation causes the first-order term to vanish. 
\[
\lim_{n \to \infty} 
\sup_{(x,t) \in \R^d \times [0,T]}
\left|
\E{
\frac{u_n\left(x+n^{-1/2}X_{\lfloor nt\rfloor+1},t\right)
-
u_n\left(x,t\right)
}{(n^{-1/2})^2}}
-
\frac{1}{2} \tr(\Sigma D^2 u_n(x,t))
\right|=0\, .
\]
Thus, if as $n\to \infty$, $u_n$ converges to some limiting function $u:\R^d\times[0, \infty) \to \R$, then in principle, $u$ should satisfy the PDE
\begin{equation}
	\begin{cases} \label{e.uheateqn0}
		\partial_t u = \frac{1}{2}\tr(\Sigma D^{2}u) &\text{in $\R^{d} \times (0,\infty),$} \\
		u(x,0) = f(x) &\text{on $\R^{d}$}\, .
	\end{cases}
\end{equation}
This is nothing but the heat equation; its 
unique solution $u=u_f$ is given by 
\begin{equation} \label{e.uformula}
	u(x,t) = \E{f(x + t^{1/2}\xi)}\, ,
\end{equation}
where $\xi \sim \normal(0,\Sigma)$. 
The function $u$ defined by \eqref{e.uformula} is 
twice continuously differentiable in space and continuously differentiable in time for $t > 0$; we reserve the notation $C^{2,1}(\R^d\times(0,\infty))$ for such functions.\footnote{In fact, $u$ is infinitely differentiable in both space and time, but we do not need this.}\ Thus, we can establish the CLT by showing that 
\eqref{eq:approximate} is indeed a good approximation of \eqref{e.uheateqn0}. By the definitions of $u_{n,f}$ and $u_f$, proving \eqref{eq:convergence_of_test} is equivalent to proving that $u_{n,f}(0,1)\rightarrow u_f(0,1)$ as $n\to \infty$ for all bounded uniformly continous $f:\R^d\to \R$. 
%(While Theorem \ref{t.clt} only requires pointwise convergence of $u_{n}$ to $u$ at the point $(0,1)$, we will in fact prove that $u_{n}\to u$ uniformly on compacts as $n\to \infty$; see Proposition \ref{prop.CLargument}.) 
This is where doubling of variables comes in.

To introduce the technique, we use a simplified version of it (which does not actually involve any doubling of variables) to prove the uniqueness of solutions to \eqref{e.uheateqn0} among $C^2$ functions which vanish at infinity. Uniqueness in fact holds within a much broader class, but proving this is more involved and less useful for the expository purposes of this article. 
\begin{proposition}\label{eq:heat_sol_comp}
Fix continuous $f:\R^d \to \R$ with $|f(x)|\to 0$ as $|x| \to \infty$, fix $T>0$, and let $u,v \in C^{2,1}(\R^d \times (0,\infty))$ be such that \[\limsup_{|x| \to \infty} \sup_{t \in [0,T]} \max(|u(x,t)|,|v(x,t)|)=0\,.
\]
If $u(x,0)=f(x)=v(x,0)$ for all $x \in \R^d$ and $\partial_t u \ge \tfrac{1}{2}\tr(\Sigma D^{2}u)$ on $\R^{d} \times (0,\infty)$ 
and $\partial_t v \le \tfrac{1}{2}\tr(\Sigma D^{2}v)$ on $\R^{d} \times (0,\infty)$, 
then $u \ge v$ on $\R^d\times[0,T]$. 
\end{proposition}
This establishes the claimed uniqueness of solutions to \eqref{e.uheateqn0} since if $\partial_t u= \tfrac{1}{2}\tr(\Sigma D^{2}u)$ 
and $\partial_t v= \tfrac{1}{2}\tr(\Sigma D^{2}v)$, then the proposition implies that $u \ge v$, and by swapping the roles of $u$ and $v$ in the proposition, we see that also $v \ge u$. 
\begin{proof}[Proof of Proposition~\ref{eq:heat_sol_comp}]
Let $T>0$ be as in the statement of the proposition, and let $\sigma:=\sup(\left(v(x,t)-u(x,t)\right)_{+}:(x,t) \in \R^d\times[0,T])$. If $\sigma=0$, then the claim is immediate, so instead assume that $\sigma>0$. Let
\[
\Psi(x,t) := v(x,t)-u(x,t)-t\cdot \sigma/(2T). 
\]
Then $\Psi(x,0)=0$ for all $x \in \R^d$, and
$\Psi^*:=\sup(\Psi(x,t):(x,t) \in\R^d\times[0,T]) \ge \sigma/2$. 

%\jlcomment{Looks like we have to write this thing as a proof by contradiction to make the claim that the sup is achieved. It parallels what we have in the proof.}
Observe that if $((x^\ell,t^\ell),\ell \ge 1)$ is any sequence of elements in $\R^d \times [0,T]$ such that $|x^\ell| \to \infty$ as $\ell \to \infty$, then $v(x^\ell,t^\ell) \to 0$ and $u(x^\ell,t^\ell) \to 0$, and thus $\limsup \Psi(x^\ell,t^\ell) \le 0$. Now suppose that $((x^{\ell}, t^{\ell}), \ell \geq 1)$ is a sequence such that $\Psi(x^\ell,t^\ell) \to \Psi^*$; then since $\Psi^* > 0$, we conclude that $(x^\ell,\ell \ge 1)$ is bounded, so there exists a subsequence of $((x^\ell,t^\ell),\ell \ge 1)$ which converges to a point $(x^*,t^*) \in \R^d \times [0,T]$. By continuity, the limit satisfies $\Psi(x^*,t^*)=\Psi^*$, so since $\Psi(x,0)=0$ for all $x$ we conclude that $t^* > 0$. 

Since $\Psi$ is maximized at $(x^*,t^*)$, necessarily $\partial_t \Psi(x^*,t^*) \ge 0$, with equality if $t^* < T$; by the definition of $\Psi$ this yields that  
\[
\partial_t v(x^*,t^*) - \partial_t u(x^*,t^*) \ge \sigma/(2T). 
\]
By the assumptions on $u$ and $v$, the preceding inequality implies that 
\[
\tr(\Sigma D^2 v(x^*,t^*))-\tr(\Sigma D^2 u(x^*,t^*)) \ge \sigma/T. 
\]
To complete the proof, note that since $\Psi$ is maximized at $(x^*,t^*)$, the Hessian of $\Psi$ in the spatial variable must be non-positive semidefinite, i.e. $D^2 \Psi(x^*,t^*) \preceq 0$.
%\lmar{Strictly speaking we should write $D^2 \Psi(x,t^*)|_{x=x^*} \preceq 0$? Do we want to say a word about this notational abuse at some point? It is quite convenient to have it...}
%\jlcomment{Could we just write $D^2 \Psi(x,t^*)|_{x=x^*}=D^{2}\Psi(x^{*}, t^*)\preceq 0$, since the equality is automatic from the regularity of $\Psi$?}
Since $\Sigma$ is positive semidefinite, it follows that $\Sigma D^2\Psi(x^*,t^*) \preceq 0$, 
so 
\[ 
\tr(\Sigma D^2\Psi(x^*,t^*))=
\tr(\Sigma D^2 u(x^*,t^*)) - \tr(\Sigma D^2 v(x^*,t^*))\le 0\, .
\]
Combined with the prior display, this contradicts the assumption that $\sigma>0$.
\end{proof}
To prove the central limit theorem, we wish to use the same style of argument as in the above proof, with $u=u_f$ the solution of the heat equation \eqref{e.uheateqn0} but with $v$ replaced by the approximation $u_n=u_{n,f}$. The idea is still to consider the location where a function of the form $u_n(x,t)-u(x,t)-ct$ is maximized, then argue that the maximum must be small. However, since the time variable in $u_n$ is effectively discrete, and that in $u$ is continuous, the cost of replacing $v$ by $u_n$ is that we need to replace the function $\Phi$ in the proof by a function with two time variables; we use a function of the form
\[
\Phi_n(x,t,s) = u_n(x,t)-u(x,s) - c_n(t+s)-C_n(t-s)^2,
\]
for suitably chosen constants $c_n,C_n$. 
The role of the additional term $C_n(t-s)^2$ is a penalization to enforce that the maximum of this function is achieved at a point where the difference $|t-s|$ is small. The details of this argument appear in the proof of Proposition~\ref{prop.CLargument}, below.

The question of whether the function $u_{n}$ approximates $u$ can also be posed in the framework of \emph{convergence of finite difference schemes}. A fundamental issue in the numerical analysis of PDEs is to construct discrete, finite difference schemes which correctly approximate the true solution of  a given PDE. Indeed, our proofs of Theorems~\ref{t.clt} and~\ref{t.prohorov} are inspired by the work of Crandall and Lions \cite[Theorem 1]{schemeshj}, where the authors obtain a quantitative estimate measuring the difference between discrete functions $u_{n}$ satisfying a monotone finite difference scheme and the viscosity solution $u$ of a first-order nonlinear Hamilton-Jacobi equation. For second-order fully nonlinear equations, the question of convergence of finite difference schemes (and not rates of convergence) was established by Barles and Souganidis \cite{BS}. In our setting, the proofs are much more straightforward than those in \cite{schemeshj}; the fact that $u_{n}$ is continuous and regular (differentiable) in the spatial argument and the linearity of the PDE both simplify the arguments. This also allows us to prove our results without introducing the machinery of viscosity solutions.

\subsection*{Related work} 
The history of the development of the CLT is described in the book \cite{MR2743162}. In addition to the classical proof based on characteristic functions, a  number of others have appeared in the literature. The paper \cite{MR3845724}, which itself provides a new proof based on expanding the random variables in the statement of the CLT with respect to the Haar basis, also describes other proofs of the CLT, including those by Berry \cite{MR3498} and Trotter \cite{MR108847}; another new, elegant proof recently appeared in \cite{MR4407621}. 
The perspective taken in this paper is inspired by the recent works \cite{MR4146542,MR4391736,https://doi.org/10.48550/arxiv.2204.03689}, which introduced the idea of using probabilistic interpretations of approximation schemes for PDEs in order to prove convergence in distribution for random processes. The most similar argument to our own that we have found in the literature is a proof of the one-dimensional CLT for $\iid$ random variables with finite third moments, by Zong and Hu \cite{MR3147846}.  Their argument is rather different from ours, and in particular, it uses as input the theory of viscosity solutions and regularity theory of parabolic PDEs, whereas ours is self-contained. However, their proof is similar to ours in that its fundamental idea is also to use the connection to the heat equation.
%The proof itself is rather different from ours, uses the theory of viscosity solutions and regularity of parabolic PDEs as an input, whereas ours is self-contained; but the approach is similar in the sense that its fundamental idea is also to use the connection to the heat equation.} \jlcomment{New attempt at last sentence: Their proof shares similarity to ours in that its fundamental idea is also to use the connection to the heat equation. However,.}

\subsection{Notation}
%Throughout the paper, we let $(X_i,i \ge 1)$ and $S_n$ be as in the introduction. We write $\cL$ for the law of $X_1$ (which is also the law of $X_i$ for each $i \ge 1$), and write $\cL_{S_n}$ for the law of $S_n$.
	Write $\mathbb{S}^d$ for the set of symmetric $d\times d$ real matrices. For functions  $u:\R^d\times[0,\infty) \to \R$, we view $[0,\infty)$ as a time dimension, and write $Du\in \R^{d}$ and $D^2u \in \mathbb{S}^d$ to denote the gradient and Hessian of $u$ with respect to space, when these are defined.

A {\em multi-index} $\al$ is a $d$-tuple of nonnegative integers $\al = (\al_{1},\al_{2},\dots,\al_{d})$. For a multi-index $\alpha$, we define
	\begin{equation*}
		|\al| = \sum_{i=1}^d \al_{i} ,\quad 
		\partial^{\al}f = \frac{\partial^{|\al|}f}{\partial x_{1}^{\al_{1}}\partial x_{2}^{\al_{2}}\dots \partial x_{n}^{\al_{n}}}.
	\end{equation*}
	When $|\alpha|=0$ we take $\partial^\alpha f=f$ by definition. 
%	\jl{For a real-valued function $u$, we reserve the notation $Du\in \R^{d}$ to denote the gradient with respect to space, and $D^{2}u\in \mathbb{S}^{d}$, the space of symmetric $d\times d$ matrices, to denote the Hessian with respect to space.}
%\lmar{Define smooth? Or define what it means to be a $C^k$ function and then only require that $f$ is $C^k$?}
For $f \in C^k(\R^d)$, we define the $C^{k}$ norm of $f$ by 
	\begin{equation}\label{eq:c4norm}
		\norm{f}_{C^{k}} = \sum_{0 \leq |\al|\leq k} \norm{\partial^{\al} f}_{\infty}. 
	\end{equation}
%	\lmar{In order for it  to be a Banach space do we need to mod out by a.e.\ equivalent functions? If so then maybe omit the mention of Banach spaces?}
%	\jlcomment{Ok, removed}
	 If $f \not\in C^k(\R^d)$ then we set $\|f\|_{C^k}=\infty$.
	Note that $\norm{f}_{C^{0}} = \norm{f}_{\infty}$ and that $\norm{f}_{C^{k}}$ is non-decreasing in $k$. 
	
	For a matrix $M\in \mathbb{S}^{d}$, we use the the notation 
	\[
	\la_{\max}(M)=\max(|\lambda|:\lambda\mbox{ is an eigenvalue of }M)\, .
	\]
%\lmar{I don't think we need both $|M|_2$ and $\lambda_{\max}$ notations, one should be enough.}

Finally, we write $\ZZ^{+} = \{0\}\cup\NN$, and let $\llb n \rrb = \{0,1,\dots,n\}$ for $n \in \ZZ^{+}$.

\section{Approximation Results for Discrete Derivatives}\label{s.lemmas}
In this section, we prove a series of lemmas which control ``discrete derivative'' expressions in terms of their true derivatives.  
First, note that if $f\in C^1(\R^d)$ then by the definition of $u=u_f$ and the dominated convergence theorem,
%\jlcomment{While I like the idea of setting $\norm{f}_{C^{1}}=\infty$ if $f\notin C^{1}$, it makes it confusing (at least to me) that we are talking about using the DCT (or here, I guess you are using Fatou, which makes it ok), but then I'm confused why we need any BUC assumptions on $f$. Also, I added $|\cdot|$ because Fatou is only true for nonnegative functions}
\begin{align*}
 \limsup_{|h| \downarrow 0} \frac{u(x+h,t)-u(x,t)}{|h|}
& 
\leq \limsup_{|h| \downarrow 0} \EE\left[\frac{|f(x+h+t^{1/2}\xi)-f(x+t^{1/2}\xi)|}{|h|}\right] \\
& \le 
\E{\limsup_{|h| \downarrow 0} \frac{|f(x+h+t^{1/2}\xi)-f(x+t^{1/2}\xi)|}{|h|}} \\
& \le \|f\|_{C^1}\, ,
\end{align*}
which implies that $\|u(\cdot,t)\|_{C^1} \le \|f\|_{C^1}$, and a similar argument shows that $\|u_n(\cdot,t)\|_{C^1} \le \|f\|_{C^1}$. 
These inequalities also automatically hold if $f \not\in C^1(\R^d)$, since in this case $\|f\|_{C^1}=\infty$.
More generally, for all integers $\ell \ge 1$ and all $t \ge 0$ we have 
%\jl{\lout{First, we observe that since $f\in \buc(\R^{d})\cap C^{\infty}$, the dominated convergence theorem applied to \eqref{e.uNdefn} and \eqref{e.uformula} yields that for every $k, \ell\in \ZZ^{+}$ and $t\geq 0$,}}
\begin{equation}\label{e.smoothbds}
\norm{u_{n}(\cdot, t)}_{C^{\ell}} \vee \norm{u(\cdot,t)}_{C^{\ell}}\leq \norm{f}_{C^{\ell}}. 
\end{equation}
This in particular implies that if $\|f\|_{C^2}<\infty$ then the Hessians $D^2u_n$ and $D^2 u$ are both defined.

%\lmar{[Old comment] Can we instead state the bounds of Lemma~\ref{lem.derivuNspacediff} as a corollary of the (new) Lemma~\ref{lem.derivuNspacediff2}? I think the proof is easier to understand in its general form than in its specific form. The bounds of Lemma~\ref{lem.derivuNspacediff} could then potentially only be stated when they are used rather than in advance (what do you all think about that idea?).}
%\lmar{[Old comment] Also this argument is conceptually simpler than that for Lemma~\ref{lem.uNspacediff} so maybe it should come first.}
We next prove a lemma which we will use to obtain bounds on finite differences of $u_n$ and its derivatives, as well as a bound on the modulus of continuity of $u$.
%\jlcomment{To be consistent with the next results, it seems to me like if $g\notin C^{2}$, the statement is still true, so maybe it's worth not putting it as an assumption in the result?}
\begin{lemma}
\label{lem.derivuNspacediff2}
If $g:\R^d\to \R$ is Borel measurable and $Y$ is a square-integrable random vector in $\R^d$ with 
$\EE[Y]=0$, $\cov(Y)=\Sigma$, and $\EE [|g(Y)|]<\infty$, then 
\[
\sup_{x \in \R^d} |\E{g(x+Y)-g(x)}| \le \|g\|_{C^2} \tr(\Sigma)/2\, .
\]
\end{lemma}
\begin{proof} If $\|g\|_{C^2}=\infty$ then the bound is automatic so we may assume $\|g\|_{C^2}<\infty$. By Taylor's theorem, we may then write 
\[
g(x+Y)-g(x) = \langle Dg(x), Y\rangle + \frac{1}{2} \langle D^2 g(x+CY)Y,Y\rangle\, ,
\]
for some $C=C(Y)\in [0,1]$; it is not hard to see that $C$ may be chosen as a Borel measurable function of $Y$. It follows that 
\begin{align*}
\E{g(x+Y)-g(x)}
& =
\EE [\langle Dg(x), Y\rangle]
+
\frac12\EE[\langle (D^2 g(x+CY))Y,Y\rangle]\\
& =  
\frac12\EE[{\langle (D^2 g(x+CY))Y,Y\rangle}]\, ,
\end{align*}
the second equality holding since $\E{Y}=0$. Since $|\langle (D^2 g(x+CY))Y,Y\rangle|\le \|g\|_{C^2}|Y|^2$, it follows that 
\[
|\E{g(x+Y)-g(x)}| \le \|g\|_{C^2} \E{|Y|^2}/2 = \|g\|_{C^2}\tr(\Sigma)/2\, .\qedhere
\]
\end{proof}
%\lmar{The statement and proof of Corollary~\ref{cor:derivuNspacediff} have been updated and should be verified.}
%\jlcomment{do we even need $f$ to be bounded and continuous?}
\begin{corollary}\label{cor:derivuNspacediff}
Fix a bounded continuous function $f:\R^d \to \R$ and write $u_n=u_{n,f}$ and $u=u_f$. Then for all $x \in \R^d$ and $t \ge 0$, we have 
\[
|u_n(x,t+n^{-1})-u_n(x,t)| 
\leq \|f\|_{C^{2}}\tr(\Sigma)/(2n)\, ,
\]
and if $f \in C^2(\R^d)$, so that $D^2u_n$ is defined, then also 
\[
\left|\tr(D^2u_n(x,t+\tmesh))-\tr(D^{2}u_n(x,t))\right|
\le 
d\cdot\norm{f}_{C^{4}}\tr(\Sigma)/(2n)\, .
\]
Finally, for all $x \in \R^d$ and $t,h \ge 0$, 
\begin{equation*}
|u(x,t+h)-u(x,t)|\leq h\norm{f}_{C^{2}}\tr(\Sigma)/2. 
\end{equation*}
\end{corollary}
%\begin{corollary}\label{cor:derivuNspacediff}
%Fix $f:\R^d \to \R$ with $\EE|f(X)|<\infty$, and define $u_{n}: \R^{d}\times(\tmesh)\ZZ^{+} \to \R$ by $u_{n}(x,k\tmesh) = \E{f\left(x + S_{k}\xmesh\right)}$. Then for all $x \in \R^d$ and $k \in \ZZ^+$
%%\[
%%		\left|\E{u_{n}(x+R\xmesh, k\tmesh) - u_{n}(x,k\tmesh)}\right| \leq \|f\|_{C^{2}}\tr(\Sigma)/(2n),
%%\]
%\[
%|u_n(x,(k+1)\tmesh)-u_n(x,k\tmesh)|
%\leq \|f\|_{C^{2}}\tr(\Sigma)/(2n)\, ,
%\]
%and for all $x \in \R^d$, $k \in \ZZ^+$ and $i,j \in \{1,\ldots,d\}$, 
%%\[
%%		\left|\E{\frac{\partial^{2}}{\partial x_{i} \partial x_{j}}u_{n}(x+X\xmesh, k\tmesh) - \frac{\partial^{2}}{\partial x_{i} \partial x_{j}}u_{n}(x,k\tmesh)}\right| \leq \norm{f}_{C^{4}}\tr(\Sigma)/(2n)\, .
%%\]
%\[
%\left|\tr(D^2u_n(x,(k+1)\tmesh))-\tr(D^{2}u_n(x,k \tmesh))\right|
%\le 
%d\cdot\norm{f}_{C^{4}}\tr(\Sigma)/(2n)\, .
%\]
%Moreover, for $u(x,t)=\EE[f(x+t^{1/2}\xi)]$, for any $h>0$, 
%\begin{equation*}
%|u(x,h)-u(x,0)|\leq h\norm{f}_{C^{2}}\tr(\Sigma)/2. 
%\end{equation*}
%\end{corollary}
\begin{proof}
Fix $t \ge 0$ and define $g(x)=u_n(x,t)$. Then by \eqref{e.smoothbds} we have
$\|g\|_{C^2} \le \|f\|_{C^2}$, 
so applying 
Lemma \ref{lem.derivuNspacediff2} with this choice of $g$ and with $Y=X_{\lfloor nt\rfloor+1}n^{-1/2}$ yields that 
\[
		\left|\E{u_{n}(x+X_{\lfloor nt\rfloor+1}\xmesh, t) - u_{n}(x,t)}\right| \leq \|f\|_{C^{2}}\tr(\Sigma)/(2n); 
\]
note that $\cov{Y}=n^{-1}\cov{X_{\lfloor nt\rfloor+1}}$. 
By \eqref{e.uNscheme}, we have $\E{u_{n}(x+X_{\lfloor nt\rfloor+1}\xmesh,t)}=u_n(x,t+\tmesh)$ and the first bound now follows.

For the second inequality, we assume $\|f\|_{C^4}<\infty$ since otherwise the bound is automatic. Fix $i,j \in \{1,\ldots,d\}$, define $g(x)=\tfrac{\partial^2}{\partial x_i \partial x_j} u_n(x,t)$, and note that $\|g\|_{C^2} \le \|f\|_{C^4}$ by the same reasoning as for \eqref{e.smoothbds}. Then again applying Lemma \ref{lem.derivuNspacediff2}, with this choice of $g$ and with $Y=X_{\lfloor nt\rfloor+1}n^{-1/2}$, we obtain 
\[
		\left|\E{\frac{\partial^{2}}{\partial x_{i} \partial x_{j}}u_{n}(x+X_{\lfloor nt\rfloor+1}\xmesh, t) - \frac{\partial^{2}}{\partial x_{i} \partial x_{j}}u_{n}(x,t)}\right| \leq \norm{f}_{C^{4}}\tr(\Sigma)/(2n)\, .
\]
Next, since $u_n(x,t+\tmesh)=\E{u_n(x+X_{\lfloor nt\rfloor+1}\xmesh,t)}$ by \eqref{e.uNscheme}, by differentiating under the integral sign we have 
\[
D^2u_n(x,t+\tmesh)-D^{2}u_n(x,t)
 = \E{D^2 u_n(x+X_{\lfloor nt\rfloor+1}\xmesh,t)-D^{2}u_n(x,t)}\, ;
\]
 the right-hand side is a matrix with 
$\E{\tfrac{\partial^2}{\partial x_i\partial x_j} (u_n(x+X_{\lfloor nt\rfloor+1}n^{-1/2},t)-u_n(x,t))}$ as its $(i,j)$ entry. (Differentiation under the integral sign  is justified by \eqref{e.smoothbds} and the assumption that $\norm{f}_{C^{4}}<\infty$.) 
%\jlcomment{This here is a place where now we are not assuming $f\in C^{4}$...}
By the preceding displayed equation, the entries of this matrix are all at most $\|f\|_{C^4}\tr(\Sigma)/(2n)$ in absolute value; since the trace is the sum of the diagonal entries, the second bound of the corollary follows. 

For the last bound, letting $\xi$ and $\xi'$ be independent and $\mathcal{N}(0,\Sigma)$-distributed, then $h^{1/2}\xi+t^{1/2}\xi'$ has the same distribution as $(h+t)^{1/2}\xi$. Since $f$ is bounded, we may thus apply Fubini's theorem to obtain 
\begin{align*}
\E{u(x+h^{1/2}\xi',t)} & 
= \E{f(x+h^{1/2}\xi'+t^{1/2} \xi)}
= \E{f(x+(h\!+\!t)^{1/2}\xi')} 
 = u(x,t\!+\!h)
\end{align*}
The final bound of the corollary then follows from Lemma~\ref{lem.derivuNspacediff2} applied with $g(x)=u(x,t)$ and $Y=h^{1/2}\xi'$, 
which has covariance matrix $h\Sigma$.
\end{proof}
We conclude this section with a lemma which allows us to formalize the Taylor expansion argument from the introduction. We also provide a quantitative version in the case when $X_i$ have a $2+\gamma$ moment.
\begin{lemma} \label{lem.uNspacediff}
Fix $f \in C^3(\R^d)$ and write $u_n=u_{n,f}$. Fix $T > 0$ and define 
\begin{equation*}
%\label{e.spaceschemeconv}
\eps_n = \sup_{(x,t)\in \R^{d} \times [0,T]}\bigg|\E{\frac{u_{n}(x + X_{\lfloor nt\rfloor+1}\xmesh , t) - u_{n}(x,t)}{n^{-1}}}
 - \frac{1}{2}\tr(\Sigma D^{2}u_{n}(x,t)) \bigg|. 
\end{equation*}
Then $\eps_n \to 0$ as $n \to \infty$, and there exists $C =C(d)>0$ such that for all $\gamma \in (0,1]$, it holds that $\e_{n} \leq C\norm{f}_{C^{3}}\E{|X_1|^{2+\ga}}n^{-\ga/2}$.
\end{lemma}

%
%Before we present the proof of Proposition \ref{prop.CLargument}, we will need a series of lemmas. This section is dedicated to those lemmas and their proofs.

\begin{proof}[Proof of Lemma \ref{lem.uNspacediff}]
Define $g:\R^d\times\R^d\times [0,T]\to \R$ by 
\begin{equation}\label{e.uNspacediffexpand}
g(h,x,t)=	u_{n}(x + h,t) - u_{n}(x,t) - \langle Du_n(x,t),h \rangle - \frac{1}{2}\langle D^{2}u_{n}(x,t)h,h\rangle. 
\end{equation}
By Taylor's theorem and \eqref{e.smoothbds}, there exists $C=C(d)$ such that for every $(x,t)\in\R^{d}\times[0,T]$,
%	, \text{  } \text{and (iii) }\frac{|g(h,x,k)|}{|h|^2} \leq C\norm{f}_{C^{3}}
	\begin{equation} \label{e.gproperties}
	|g(h,x,t)|
	\leq C\|u_n\|_{C^3}(|h|^{3}\wedge |h|^{2})
	\leq C\norm{f}_{C^{3}}(|h|^{3}\wedge |h|^{2})
	\leq C\norm{f}_{C^3}(|h|^{2+\gamma}\wedge |h|^2).
	\end{equation}
	
	In this paragraph, it is convenient to write $X=X_{\lfloor nt\rfloor+1}$ for simplicity. 
	Substituting $h =X\xmesh$ into \eqref{e.uNspacediffexpand} yields 
\begin{align*}
g(Xn^{-1/2}&,x,t)\\
& = 
u_{n}(x+X\xmesh, t) - u_{n}(x,t)
-
\frac{\langle Du_{n}(x,t), X\rangle}{n^{1/2}} - \frac{\langle D^2 u_{n}(x,t)X, X\rangle}{2n}.
\end{align*}
Note that $\EE[\langle Du_{n}(x,t), X\rangle]=0$ since $\EE[X]=0$. Dividing both sides by $n^{-1}$ and taking expectations, we thus obtain
\begin{align*}
\E{\frac{g(Xn^{-1/2},x,t)}{n^{-1}}}
&=
\E{\frac{u_{n}(x+X\xmesh, t) - u_{n}(x,t)}{n^{-1}}}
 - \E{\frac{\langle D^2 u_{n}(x,t)X, X\rangle}{2}}\\
&=\E{\frac{u_{n}(x+X\xmesh, t) - u_{n}(x,t)}{n^{-1}}}
- \frac{1}{2}\tr(\Sigma D^{2}u_{n}(x,t)),
\end{align*}
and this equality and the fact that $X\eqdist X_1$ together imply that 
\begin{equation}\label{eq:error}
		\e_{n} \leq \sup_{(x,t)\in \R^{d} \times [0,T]}\EE\bigg[\left|\frac{g(X_1\xmesh,x,t)}{n^{-1}}\right|\bigg]. 
\end{equation}
If $\E{|X_1|^{2+\gamma}}<\infty$ for some value $\gamma \in (0,1]$, then using that $|g(h,x,t)|\le C\|f\|_{C^3}|h|^{2+\gamma}$ from \eqref{e.gproperties}, the bound \eqref{eq:error} yields that 
\[
	\eps_n	\leq \sup_{(x,t)\in \R^{d} \times [0,T]} \E{C\norm{f}_{C^{3}}|X_1|^{2+\ga}n^{-\ga/2}} 
	= C\norm{f}_{C^3}\E{|X_1|^{2+\ga}}n^{-\ga/2}\,,
\]
and this yields the second claim of the result.

%	\begin{multline*}
%		u_{n}(x+X\xmesh, k\tmesh) - u_{n}(x,k\tmesh) \\ =  \xmesh \langle Du_{n}(x,k\tmesh), X\rangle + \frac{(\xmesh)^{2}}{2}\langle D^2 u_{n}(x,k\tmesh)X, X\rangle + g(X\xmesh,x,k).
%	\end{multline*}
%	Dividing both sides by $n^{-1}$, taking an expectation, and using the fact that $\EE[X]=0$, gives
%	\begin{align*} \label{e.texpandafterexp}
%		&\EE\left[\frac{1}{(\xmesh)^2}\left(u_{n}(x+X\xmesh, s) - u_{n}(x,k\tmesh)\right)\right] \notag \\
%		&= \EE\left[\frac{1}{\xmesh}\langle Du_{n}(x,k\tmesh),X\rangle + \frac{1}{2}\langle D^2u_{n}(x,k\tmesh)X, X \rangle + \frac{g(X\xmesh)}{(\xmesh)^2}\right] \notag \\
%		&= \frac{1}{2}\tr\left(\Sigma D^2u_{n}(x,k\tmesh)\right) + \EE\left[\frac{g(X\xmesh,x,k)}{(\xmesh)^2}\right].
%	\end{align*}

To complete the proof of the first claim, it suffices to show that in all cases 
	\begin{equation} \label{e.glimit}
		\lim_{n\to\infty}\sup_{(x,t)\in \R^{d} \times [0,T]}\EE\bigg[\left|\frac{g(X_1\xmesh,x,t)}{n^{-1}}\right|\bigg] = 0.
	\end{equation}
	To this end, fix $\de > 0$ and choose a sequence $((x_{n}^{\de},t_{n}^{\de}),n\geq 1)$ of elements of $\R^d \times [0,T]$ such that for all $n \ge 1$, 
	\begin{equation*}
		\sup_{(x,t)\in \R^{d} \times [0,T]}\EE\bigg[\left|\frac{g(X_1\xmesh,x,t)}{n^{-1}}\right|\bigg]\leq  \E{\left|\frac{g(X_1\xmesh,x_{n}^{\de},t_{n}^{\de})}{n^{-1}}\right|} + \de.
	\end{equation*}
Note that by the bound $|g(h,x,t)| \le C\|f\|_{C^3}|h|^{2}$ from \eqref{e.gproperties}, we have
	\begin{equation*}
		\left|\frac{g(X_1\xmesh,x_{n}^{\de},t_{n}^{\de})}{n^{-1}}\right| 
		\leq \frac{C\|f\|_{C^3}|X_1\xmesh|^{2}}{n^{-1}}
		= C\norm{f}_{C^{3}}|X_1|^{2}.
	\end{equation*}
	Because $\E{|X_1|^{2}} < \infty$, we may thus apply the dominated convergence theorem,  then use the bound $|g(h,x,t)| \le C\norm{f}_{C^{3}}|h|^{3}$ from \eqref{e.gproperties}, to obtain
	\begin{align*}
		\lim_{n\to\infty}\sup_{(x,t)\in \R^{d} \times \llb 2n \rrb}\EE\bigg[\left|\frac{g(X_1\xmesh,x,t)}{n^{-1}}\right|\bigg] &\leq \limsup_{n\to\infty} \EE\bigg[\left|\frac{g(X_1\xmesh,x_{n}^{\de},t_{n}^{\de})}{n^{-1}}\right|\bigg] + \de\\
		&\leq\E{\limsup_{n\to\infty}\left|\frac{g(X_1\xmesh,x_{n}^{\de},t_{n}^{\de})}{n^{-1}}\right|} +\de \\
		&\leq \E{\limsup_{n\to\infty}C\norm{f}_{C^{3}}|X_1|^{3}{n^{-1/2}}} +\de\\
		&=\de,
	\end{align*}
	where for the last equality we used that $|X_1|^{3}\xmesh  \xrightarrow{a.s.} 0$. Since $\delta > 0$ was arbitrary, this establishes \eqref{e.glimit} and completes the proof.
\end{proof}

\section{Doubling of variables}\label{s.bigprop}
In this section, we carry out the doubling of variables argument discussed in the introduction. The following proposition almost establishes Theorem~\ref{t.clt}, but for a slightly more restrictive class of test functions $f$.
\begin{proposition} \label{prop.CLargument}
Fix $f\in C^4(\R^d)$ with $|f(x)| \to 0$ as $|x| \to \infty$. Then writing $u_n=u_{n,f}$ and $u=u_f$, we have 
%	Let $f \in C^{\infty}(\R^{d})$ with 
%	\begin{equation} \label{e.fdecay}
%		|f(x)| \to 0 \text{ as } |x| \to \infty,
%	\end{equation}
%	and $\norm{f}_{C^{4}} < \infty$. Let $u_N(x,k\tmesh)=\E{f(x+S_kN^{-1/2})}$ as in \eqref{e.uNdefn}, and let $u$ solve the heat equation 
%\begin{equation*}
%	\begin{cases}
%		\partial_t u = \frac{1}{2}\tr(\Sigma D^{2}u) &\text{on $\R^{d} \times (0,\infty)$} \\
%		u(x,0) = f(x) &\text{on $\R^{d}$}.
%	\end{cases}
%\end{equation*}
%Then
%	\begin{equation} \label{e.convergence}
%		\sup_{(x,k) \in \R^{d}\times\llb2N \rrb}|u_{N}(x,k\tmesh) - u(x,k\tmesh)| \to 0 \text{ as } N \to \infty. 
%	\end{equation}
	\begin{equation} \label{e.convergence}
		\sup_{(x,t) \in \R^{d}\times[0,2]}|u_{n}(x,t) - u(x,t)| \to 0 \text{ as } n \to \infty. 
	\end{equation}
Moreover, there exists $C=C(d,\|f\|_{C^4},\lambda_{\max}(\Sigma))$ such that for all $\gamma \in (0,1]$
	\[
	\sup_{(x,t) \in \R^{d}\times[0,2]}|u_{n}(x,t) - u(x,t)| \le C \E{|X_1|^{2+\gamma}} n^{-\gamma/2}\, .
	\]
	
	\end{proposition}
Recall that we write $\llb n \rrb = \{0,1,\ldots,n\}$ for $n \in \ZZ_+$. Let 
\[
\sig_{n} = \sup_{(x,k) \in \R^{d}\times\llb 2n \rrb}(u_{n}(x,k\tmesh) - u(x,k\tmesh))_+\, ,
\] 
and note that $\sig_n \le \|u_n\|_\infty+\|u\|_\infty \le 2 \|f\|_\infty$. 
Let $c_n=\sigma_n/8$, let $C_n=2\|f\|_\infty n^{1/2}$, and define
$\phi_{n}: \R^{d}\times \llb 2n \rrb\times [0,2] \to \R$ by 
\begin{equation*}
\phi_n(x,k,s) = \\ u_{n}(x,k\tmesh) - u(x,s) - c_n(k\tmesh + s) -C_n(k\tmesh-s)^2\, .
\end{equation*}

The following lemma will be used in the proof.
\begin{lemma}\label{lem:supremum_info}
For $\phi_{n}$ defined above, $\sup_{(x,k,s) \in \R^d\times \llb 2n\rrb \times [0,2]} \phi_{n}(x,k,s)>\sigma_{n}/2$. Also, there exists $C=C(d,\|f\|_{C^2},\lambda_{\max}(\Sigma))> 0$ such that for all $n$, for any point $(x_{0}, k_{0}, s_{0})\in \R^d\times \llb 2n\rrb \times [0,2]$ such that 
\begin{equation}\label{2eq:supremum_point}
\phi_n(x_0,k_0,s_0)=\sup_{(x,k,s) \in \R^{d}\times \llb 2n \rrb\times [0,2]} \phi_n(x,k,s),
\end{equation}
it holds that $|k_0\tmesh-s_0| \le Cn^{-1/2}$. Finally, if $\sigma_{n}>0$, then there exists a point $(x_0,k_0,s_0) \in \R^d\times \llb 2n\rrb \times [0,2]$ such that the supremum is achieved.  
\end{lemma}
In the coming proofs, for a function $a: \mathbb{N}\to \RR$, we write $a(n)=O(n^{-1/2})$ to mean there exists $C=C(d,\|f\|_{C^2},\lambda_{\max}(\Sigma))>0$ such that $\sup_{n \in \NN} n^{1/2}|a(n)| \le C$.
\begin{proof}
First note that 
\begin{align*}
\sup_{(x,k,s) \in \R^{d}\times \llb 2n \rrb\times [0,2]} 
\phi_n(x,k,s) 
& \ge 
\sup_{(x,k) \in \R^{d}\times \llb 2n \rrb}
\phi_n(x,k,k\tmesh)\\
& = 
\sup_{(x,k) \in \R^{d}\times \llb 2n \rrb}
\left(
u_{n}(x,k\tmesh) - u(x,k\tmesh) - 2c_n k\tmesh \right)\\
& \ge \frac{\sig_n}{2}\, ,
\end{align*}
%\sout{+ 6M}\sout{6M+}
the last inequality holding by the definition of $\sig_n$ and since for all $k \in \llb 2n \rrb$ we have $2c_nk\tmesh = (\sig_n/8)2k\tmesh \le \sigma_n/2$. 

Now let $(x_0,k_0,s_0)$ be any point achieving the supremum in \eqref{2eq:supremum_point}. We bound $|k_0\tmesh-s_0|$ in cases, depending on the value of $s_0$. 

First suppose that $s_0=0$. If $k_0=0$ then  $s_0=k_0\tmesh$. Otherwise, the fact that $\phi_n(x_0,k_0-1,s_0) \le \phi_n(x_0,k_0,s_0)$ may be rewritten (using that $s_0=0$) as 
\[
u_n(x,(k_0-1)\tmesh)  - C_n ((k_0-1)\tmesh)^2
\le
u_n(x,k_0\tmesh)-c_n \tmesh - C_n (k_0\tmesh)^2\, .
\]
Rearranging this, and using that $(k_0\tmesh)^2-((k_0-1)\tmesh)^2 \ge k_0 n^{-2}$, 
it follows that
\begin{align*}
C_n\cdot  k_0\tmesh & \le \frac{u_n(x,k_0\tmesh)-u_n(x,(k_0-1)\tmesh)}{\tmesh} - c_n  
 \le
\frac{\|f\|_{C^2} \tr(\Sigma)}{2}\, ,
\end{align*}
where for the last inequality we have used Corollary~\ref{cor:derivuNspacediff} and the fact that $c_{n}\geq 0$. It follows that $k_0 \tmesh-s_0=k_0\tmesh \le \|f\|_{C^2}\tr(\Sigma)/(2C_n)=O(n^{-1/2})$, 

Next suppose that $s_0 > 0$. 
Since $u\in C^{2,1}(\R^d\times(0,\infty))$, by the definition of $\phi_n$ we have 
\[
\partial_s\phi_n(x,k,s) = -\partial_s u(x,s) - c_n + 2C_n(k\tmesh-s).
\]
Since $(x_0,k_0,s_0)$ is a maximum of $\phi_n$ on $\R^d\times \llb 2n\rrb \times [0,2]$, it follows that if $s_0 \in (0,2)$ then $\partial_s(\phi_n(x,k,s))|_{(x,k,s)=(x_0,k_0,s_0)} = 0$, and if $s_0=2$ then this partial derivative is at least zero.
% $\partial_s(\phi_n(x,k,s))|_{(x,k,s)=(x_0,k_0,s_0)} \ge 0$. 
Since $\partial_t u = \tfrac12\tr(\Sigma D^2 u)$ in $\R^d\times (0,\infty)$, if $s_0 \in (0,2)$, we obtain that 
\[
2C_n(k_0\tmesh-s_0) = c_n +\tfrac12\tr(\Sigma D^2 u)\, ,
\]
so 
\[
|k_0\tmesh-s_0| \leq \frac{1}{2C_{n}}\cdot \left(c_n+\frac12 |\tr(\Sigma D^2 u)|\right) =O(n^{-1/2})\, .
\]
If instead $s_0=2$, then using that $\partial_s(\phi_n(x,k,s))|_{(x,k,s)=(x_0,k_0,s_0)} \ge 0$ we similarly obtain that 
\[
k_0\tmesh \ge s_0-\frac{1}{2C_{n}}\cdot \left(c_n+\frac12 |\tr(\Sigma D^2 u)|\right)
=s_0-O(n^{-1/2})\, ,
\] 
so since also $k_0\tmesh \le 2=s_0$ we again conclude that $|k_0\tmesh-s_0|=O(n^{-1/2})$. 

To complete the proof of the lemma, suppose $((x^\ell,k^\ell,s^\ell),\ell \ge 1)$ is a sequence of elements of $\R^{d}\times \llb 2n \rrb\times [0,2]$. 
If $x^\ell \to \infty$ as $\ell \to \infty$, then since $|f(x)| \to 0$ as $x \to \infty$, and $u_n(x,k\tmesh)=\E{f(x+S_k\xmesh)}$ and $u(x,t)=\E{f(x+t^{1/2}\xi)}$ where $\xi \sim \normal(0,\Sigma)$, 
it follows that as $\ell \to \infty$,
\[
\sup_{k \in \llb 2n \rrb} |u_n(x^\ell,k\tmesh)| \to 0\quad\mbox{ and }\quad 
\sup_{t \in [0,2]} |u(x^\ell,t)| \to 0.
\]
Thus, whenever $x_\ell \to \infty$ as $\ell \to \infty$, we have 
%\[
%\limsup_{\ell \to \infty} \phi_n(x^\ell,k^\ell,s^\ell) \le 6M. 
%\]
\begin{equation*}
\limsup_{\ell \to \infty} \phi_n(x^\ell,k^\ell,s^\ell) \le 0.
\end{equation*}
On the other hand, choosing $((x^\ell,k^\ell,s^\ell),\ell \ge 1)$ so that 
\[
\phi_n(x^\ell,k^\ell,s^\ell) \to \sup_{(x,k,s) \in \R^{d}\times \llb 2n \rrb\times [0,2]}
\phi_n(x,k\tmesh,s) \, ,
\]
if $\sigma_n > 0$ then, since the last supremum is at least $\sig_n/2$, it follows that 
%\[
%\liminf_{\ell \to \infty} \phi_n(x^\ell,k^\ell,s^\ell) > 6M,
%\]
\[
\liminf_{\ell \to \infty} \phi_n(x^\ell,k^\ell,s^\ell) > 0,
\]
so $(x^\ell,\ell \ge 1)$ must be a bounded sequence. Since $\llb 2n \rrb$ is finite, $[0,2]$ is compact and $\phi_n$ is continuous in $x$ and $s$, it follows that, after passing to a subsequence, $((x^\ell,k^\ell,s^\ell),\ell \ge 1)$ converges to some $(x_0,k_0,s_0)\in \R^{d}\times \llb 2n \rrb\times [0,2]$ as $\ell \to \infty$ and that the supremum is achieved at the limit. 
\end{proof}

Equipped with this lemma, we are now ready to prove the main result of this section.

\begin{proof}[Proof of Proposition~\ref{prop.CLargument}]

%In the proof, we sometimes write $\sup=\sup_{(x,k,s) \in \R^{d}\times \llb 2n \rrb\times [0,2]}$ for readability. 
Recall that 
\[
\sig_{n} = \sup_{(x,k) \in \R^{d}\times\llb 2n \rrb}(u_{n}((x,k\tmesh) - u(x,k\tmesh))_+\,
\] 
and let 
$\tilde{\sig}_{n} = \sup_{(x,k) \in \R^{d}\times\llb 2n \rrb}(u(x,k\tmesh) - u_n(x,k\tmesh))_+$, so that 
\[
\sup_{(x,k) \in \R^{d}\times\llb 2n \rrb}|u_{n}(x,k\tmesh) - u(x,k\tmesh)|=\max(\sig_n,\tilde{\sig}_n). 
\]
Since $u_{n}(x, \cdot)$ is constant on time intervals of the form $[kn^{-1}, (k+1)n^{-1})$, we have 
\begin{align*}
\sup_{(x,t) \in \R^d \times[0,2] }
(u_{n}((x,t) - u(x,t))_+
& \le \sigma_n + \sup_{x \in \R^d, |t-s| \le n^{-1}} |u(x,t)-u(x,s)| \\
& \le \sigma_n +\norm{f}_{C^{2}}\tr(\Sigma)/(2n),
\end{align*}
the last inequality holding by the final bound of Corollary~\ref{cor:derivuNspacediff}. It likewise holds that 
\[
\sup_{(x,t) \in \R^d \times[0,2] }
(u(x,t)-u_{n}((x,t))_+
\le \tilde{\sigma}_n+\norm{f}_{C^{2}}\tr(\Sigma)/(2n),
\]
so to prove the proposition, it suffices to show that $\limsup_{n \to \infty} \sig_n = 0$ and $\limsup_{n \to \infty} \tilde{\sig}_n= 0$ (or equivalently for the quantitative estimate, to show $\max\left(\sigma_{n}, \tilde{\sig}_{n}\right)\leq C \E{|X_1|^{2+\gamma}} n^{-\gamma/2}$). We prove this only for $\sig_n$ as the bound for $\tilde{\sig}_n$ follows by a symmetric argument. 

Fix $n$ with $\sigma_n > 0$; by Lemma~\ref{lem:supremum_info}, there exists $(x_0,k_0,s_0) \in \R^d\times \llb 2n\rrb \times [0,2]$ such that 
$\phi_n(x_0,k_0,s_0)=\sup \phi_n(x,k,s)$. The proof breaks into cases depending on where the supremum is achieved. 
\smallskip

\noindent {\emph{{\color{blue} Case 1}: $k_0 > 0, s_0 > 0$.}}

By the definition of $\phi_n$, and since $u(x,s)$ is smooth for $s > 0$, we have 
\[
\partial_s\phi_n(x,k,s) = -\partial_s u(x,s) - c_n + 2C_n(k\tmesh-s).
\]
Since $(x_0,k_0,s_0)$ is a maximum of $\phi_n$ on $\R^d\times \llb 2n\rrb \times [0,2]$, by the same reasoning as in the prior lemma, we conclude that
\begin{equation}\label{2eq:perturb_s_bound}
c_n \le 2C_n(k_0\tmesh-s_0) - \frac{1}{2}\tr(\Sigma D^{2}u(x_0,s_0)).
\end{equation}

We now perform a similar ``maximizing'' argument for $k_0$; the fact that $\phi_n(x_0,k_0,s_0) \ge \phi_{n}(x_0,k_{0}-1, s_0)$ may be rewritten as
\begin{multline*}
u_{n}(x_{0},k_0\tmesh) - c_n(k_0\tmesh) - C_n(k_0\tmesh-s_0)^2\\
 \ge 
 u_{n}(x_{0},(k_{0}-1)\tmesh) - c_n((k_{0}-1)\tmesh) - C_n((k_{0}-1)\tmesh-s_0)^2\, .
\end{multline*}
Rearranging and dividing by $\tmesh$ yields
\begin{align} \label{2eq:perturb_k_bound}
c_n
& \le 
\frac{u_{n}(x_{0},k_0\tmesh)-u_{n}(x_{0},(k_0-1)\tmesh)}{\tmesh} \notag \\
& \quad + C_n\frac{1}{\tmesh}\bigl(((k_0-1) \tmesh-s_0)^2-(k_0\tmesh-s_0)^2\bigr) \notag\\
& = 
\frac{u_{n}(x_{0},k_0\tmesh)-u_{n}(x_{0},(k_0-1)\tmesh)}{\tmesh} 
-2C_n (k_0\tmesh-s_0) + \frac{C_n}{n}\, .
\end{align}
Adding this inequality to \eqref{2eq:perturb_s_bound} and using that $2c_n=\sigma_n/4$, we obtain that 
\[
\frac{\sigma_n}{4} \le 
\frac{u_{n}(x_{0},k_0\tmesh)-u_{n}(x_{0},(k_0-1)\tmesh)}{\tmesh} 
-\frac{1}{2}\tr(\Sigma D^{2}u(x_{0},s_0)) + \frac{C_n}{n}\, .
\]
Using the equality $u_n(x,k\tmesh)=\E{u_n(x+X\xmesh,(k-1)\tmesh)}$ and 
applying Lemma~\ref{lem.uNspacediff}, we have that 
\[
\left|
\frac{u_{n}(x_{0},k_0\tmesh)-u_{n}(x_{0},(k_0-1)\tmesh)}{\tmesh}
- \frac{1}{2}\tr(\Sigma D^2 u_n(x_{0},(k_0-1)\tmesh))
\right| \le \eps_n,
\]
so 
\begin{equation}\label{2eq:sigman_bd1}
\frac{\sigma_n}{4} \le 
\frac{1}{2}\left(\tr(\Sigma D^2 u_n(x_{0},(k_0-1)\tmesh))-\tr(\Sigma D^{2}u(x_0,s_0))\right) + \frac{C_n}{n}+\eps_n\, .
\end{equation}
To bound the difference of traces, we first note that since $\phi_n$ is maximized at $(x_0,k_0,s_0)$, its Hessian is non-positive semidefinite at this point: 
\[
D^2 \phi_n(x_0,k_0,s_0) = D^2u_n(x_0,k_0 \tmesh)-D^{2}u(x_0,s_0) \preceq 0. 
\]
Since $\Sigma$ is a covariance matrix, it is positive semidefinite, and therefore $\Sigma \big(D^2u_n(x_0,k_0 \tmesh)-D^2u(x_0,s_0)\big) \preceq 0$, and so the linearity of the trace implies 
\begin{equation*}
\tr\bigl(\Sigma D^2 u_n(x_0,k_0 \tmesh)\bigr) \leq \tr\bigl(\Sigma D^2u(x_0,s_0)\bigr)\,.
\end{equation*}
Therefore,
\begin{align*}
\tr\big(\Sigma D^2 u_n&(x_0,(k_0-1)\tmesh)\big)-\tr\big(\Sigma D^{2}u(x_0,s_0)\big)\\
&
\le 
\tr\big( \Sigma \big(D^2u_n(x_0,(k_0-1) \tmesh)-D^{2}u_n(x_0,k_0 \tmesh)\big)\big)\\
& \le \lambda_{\max}(\Sigma) \left|\tr \Bigl(D^2u_n(x_0,(k_0-1)\tmesh)-D^{2}u_n(x_0,k_0 \tmesh)\Bigr)\right|\\
& \le 
\lambda_{\max}(\Sigma)\cdot d \cdot \|f\|_{C^4} \tr(\Sigma)/(2n)\\
& \leq d^{2}\cdot \lambda_{\max}(\Sigma)^{2}\cdot \|f\|_{C^4}/(2n)\, ,
\end{align*}
where in the second-to-last line we have used the second bound of Corollary~\ref{cor:derivuNspacediff}, and in the last line we have used that $\tr(\Sigma) \le d\lambda_{\max}(\Sigma)$. 
Combining this with \eqref{2eq:sigman_bd1}, we obtain that 
\[
\frac{\sigma_n}{4} \le 
\frac{d^2\lambda_{\max}(\Sigma)^2 \|f\|_{C^4}}{4n} 
 + \frac{C_n}{n}+\eps_n\, .
\]
Since $C_n=2\|f\|_{\infty} n^{1/2}$, the right-hand side tends to $0$ as $n \to \infty$; also, by the bound on $\eps_n$ from Lemma~\ref{lem.uNspacediff}, for any $\gamma \in (0,1]$ the right hand-side is at most $C\E{|X_1|^{2+\gamma}||}n^{-\gamma/2}$ for some constant $C=C(d,\|f\|_{C^4},\lambda_{\max}(\Sigma))$. (For this we are using that $\|f\|_{C^k}$ is increasing in $k$ and that $\|f\|_{C^0}=\|f\|_{\infty}$.)
\smallskip

%In the remaining cases, 
%we use that, by Lemma~\ref{lem:supremum_info}, if
%\[
%\phi_n(x_0,k_0,s_0)=\sup_{(x,k,s) \in \R^{d}\times \llb 2n \rrb\times [0,2]} \phi_n(x,k,s)>0
%\] 
%and either $k_0=0$ or $s_0=0$, then both $k_0\tmesh <2$ and $s_0 < 2$. \smallskip

\noindent {\emph{{\color{blue} Case 2}: $k_0 = 0, s_0 \geq 0$.}}
Recall from Lemma~\ref{lem:supremum_info} that $\phi_n(x_0,k_0,s_0) \ge \sig_n/2$. Using the definition of $\phi_n$, we thus have 
\begin{align*}
\sig_n/2 
\le \phi_n(x_0,0,s_0) \le u_n(x_0,0)-u(x_0,s_0) &= f(x_0)-\EE[f(x_0+s_0^{1/2}\xi)]\\
& \le \left|\E{f(x_0+s_0^{1/2}\xi)-f(x_0)}\right|. 
\end{align*}
%\jlcomment{Should this argument maybe come closer to Lemma \ref{lem.derivuNspacediff2}? It's short enough it's ok ``in proof,'' although we could maybe put it in the statement of Corollary \ref{cor:derivuNspacediff}}
%\lmar{Agree and have made the change.}
By the last bound of Corollary~\ref{cor:derivuNspacediff}, we have
\[
\Big|\E{f(x_0+s_0^{1/2}\xi)-f(x_0)}\Big|
= |u(x_0,s_0)-u(x_0,0)|
 \le s_0\|f\|_{C^2} \tr(\Sigma)/2\, ,
\]
which combined with the preceding inequality yields that 
\[
\sig_n \le s_0 \|f\|_{C^2}\tr(\Sigma)\, .
\]
By Lemma~\ref{lem:supremum_info}, we have 
$s_0=|s_0-k_0n^{-1}|=O(n^{-1/2})$, so the preceding bound implies that $\sig_n=O(n^{-1/2})$.
\smallskip

\noindent {\emph{{\color{blue} Case 3}: $k_0 > 0, s_0 = 0$.}} The argument is similar to that from Case 2. Using that $\phi_n(x_0,k_0,s_0) \ge \sig_n/2$ and the definition of $\phi_n$ gives us 
\begin{align*}
\sig_n/2 \le \phi_n(x_0,k_0,0)&\le u_n(x_0,k_0\tmesh)-u(x_0,0)  \\
& = u_n(x_0,k_0\tmesh)-u_n(x_0,0) \\
& \le 
k_0 \max_{1 \le k \le k_0}|u_n(x_0,k\tmesh)-u_n(x_0,(k-1)\tmesh)|\,.
%& \le 
%(1/3)^{1/2}n^{3/4}\max_{1 \le k \le k_0}|u_n(x_0,k\tmesh)-u_n(x_0,(k-1)\tmesh)|\, ,
\end{align*}
Using that $u_n(x,k\tmesh)=\E{u_n(x+X\xmesh,(k-1)\tmesh)}$ and 
applying Corollary~\ref{cor:derivuNspacediff}, it follows that for all $1 \le k \le k_0$, 
\begin{equation} \label{2eq:uN_timestep_diff}
	|u_n(x_0,k\tmesh)-u_n(x_0,(k-1)\tmesh)| \le \|f\|_{C^{2}}\tr(\Sigma)/(2n)\, ,
\end{equation}
which together with the previous bound implies that
\[
\sig_n \le k_0\|f\|_{C^{2}}\tr(\Sigma)n^{-1}.
\]
By Lemma~\ref{lem:supremum_info} we have that 
$k_0 n^{-1}=|k_0n^{-1}-s_0|=O(n^{-1/2})$, so this bound again yields that $\sig_n=O(n^{-1/2})$.
\end{proof}
Theorem~\ref{t.prohorov} is an immediate consequence of the second bound of Proposition~\ref{prop.CLargument}, so all that remains for us to do is to deduce Theorem~\ref{t.clt} from the proposition. 
%\begin{remark}
%	In the case that $\E{|X|^{3}} < \infty$, by using the rate of convergence from Remark \ref{rem.3rdmotaylorrate} in place of Lemma \ref{lem.uNspacediff}, a more careful analysis of Case 1 yields 
%	\begin{equation*}
%		\sig_{n} \leq C(n^{-1/2} + \norm{f}_{C^{3}} \E{|X|^{3}} n^{-1/2} + |\Sigma|\norm{f}_{C^{4}}n^{-1}).
%	\end{equation*}
%	Therefore, in all three cases it can be shown that $\sig_{n} \leq \tilde{C}n^{-1/2}$, where $\tilde{C} > 0$ depends on $\norm{f}_{C^{4}}, |\Sigma|,$ and $\E{|X|^{3}}$.
%\end{remark}

%%%%%%%%%%%%%%%%%

\section{Proof of Theorem~\ref{t.clt}}
For any bounded continuous function $f:\R^d \to \R$, there exists an approximating sequence of functions $f^{(k)}:\R^d \to \R$, all bounded and uniformly continuous, with compact support and with $\|f^{(k)}\|_{\infty} \le \|f\|_{\infty}$, such that $f^{(k)}$ and $f$ agree on $[-k,k]^d$. For such functions, we have 
\begin{align*}
%\limsup_{k \to \infty} \limsup_{n \to \infty}
\EE\left[|f(S_n n^{-1/2})-f^{(k)}(S_n n^{-1/2})|\right]
& 
\le 2\|f\|_\infty \P{|S_n|_\infty \ge kn^{1/2}}\, ,
\end{align*}
where for $x\in \RR^{d}$, $|x|_{\infty}=\max_{1\leq i\leq d} |x_{i}|$, so by Chebyshev's inequality,
\[
\limsup_{k\to \infty} \limsup_{n \to \infty}
\big|\E{f(S_n n^{-1/2})}-\E{f^{(k)}(S_n n^{-1/2})}\big|=0\, .
\]
To prove Theorem~\ref{t.clt} it thus suffices to prove that for all bounded, uniformly continuous $f:\R^d \to \R$ with compact support, we have $\EE[f(S_n n^{-1/2})] \to \EE[f(\xi)]$. For the remainder of the proof, we fix such a function $f$; we will use the smoothing effect of convolution with a Gaussian in order to be able to apply Proposition~\ref{prop.CLargument}.

Write  $u=u_f$ and $u_n=u_{n,f}$, like in the rest of the paper. 
Using $\eta_t$ for the probability density function of $t^{1/2}\xi$, then 
\[
u(x,t) = \int_{\R^d} f(x + t^{1/2}y) \eta_t(y)dy\, ;
\]
since $\eta_t \in C^{\infty}(\R^d)$ and $f$ is bounded, it follows by differentiation under the integral sign that $u(\cdot,t) \in C^{\infty}(\R^d)$ for all $t > 0$. Since $f$ has compact support and $\eta_t(y) \to 0$ as $|y|\to \infty$, this identity likewise implies that $u(x,t) \to 0$ as $|x| \to \infty$. Writing $f^t(x):= u(x,t)$, it thus follows that for every $t>0$, $f^t$ satisfies the assumptions of Proposition~\ref{prop.CLargument}; this in particular implies that 
\[
u_{n,f^t}(0,1) \to u_{f^t}(0,1)
\]
as $n \to \infty$. Rewritten using the definitions of $u_{n,f^t}$ and $u_{f^t}$, this convergence states that 
\[
\E{f(S_n n^{-1/2}+t^{1/2}\xi)} \to 
\E{f(\xi + t^{1/2}\xi')}\, 
\]
as $n \to \infty$, where $\xi,\xi'$ are independent and $\mathcal{N}(0,\Sigma)$-distributed (and $\xi$ is independent of $S_n$). Finally, since $f$ is bounded and uniformly continuous, for any $\eps > 0$ we may choose $t>0$ small enough that 
$\E{|f(x+t^{1/2}\xi)-f(x)|} \le \eps$ for all $x \in \R^{d}$, and the triangle inequality then implies that 
\begin{align*}
 \limsup_{n \to \infty}\, & \big|\E{f(S_n n^{-1/2})}-\E{f(\xi)}\big| \\
& \le 2\eps + \limsup_{n \to \infty}
\big|\E{f(S_n n^{-1/2}+t^{1/2}\xi)} -
\E{f(\xi + t^{1/2}\xi')}\big|\\
& = 2\eps
\, .
\end{align*}
Since $\eps > 0$ was arbitrary, this completes the proof.

\subsection*{Acknowledgements}
LAB was partially supported by NSERC Discovery Grant 643473. GB was partially supported by an NSERC Canada Graduate Doctoral Scholarship and NSERC Discovery Grant 247764. EB was partially supported by NSERC Discovery Grants 247764 and 643473. JL was partially supported by NSERC Discovery Grant 247764, FRQNT Grant 250479, and the Canada Research Chairs program.

\bibliographystyle{abbrv} 
\bibliography{bib}
\end{document}